\documentclass[a4paper,12pt]{amsart}
\usepackage[utf8]{inputenc}
\usepackage{fullpage}
\usepackage[english]{babel}
\usepackage{enumerate}
\usepackage[usenames,dvipsnames]{xcolor}
\usepackage[colorlinks=true,linkcolor=Blue,citecolor=Green]{hyperref}
\usepackage{bookmark}
\usepackage{amsmath,amsthm,amsfonts,amssymb}
\usepackage{graphicx}

\newtheorem{theorem}{Theorem}[section]
\newtheorem{proposition}[theorem]{Proposition}
\newtheorem{cor}[theorem]{Corollary}
\newtheorem{lemma}[theorem]{Lemma}

\newcommand{\C}{\mathcal{C}}
\newcommand{\V}{\mathrm{Vol}}
\renewcommand{\Re}{{\mathbb R}}
\newcommand{\Red}{\Re^d}
\newcommand{\Ci}{\mathcal{C}}
\newcommand{\Ei}{\mathcal{E}}
\newcommand{\Ball}{\mathbf{B}^d}
\newcommand{\omegad}{\omega_d}
\newcommand{\st}{\;|\;}

\newcommand{\noshow}[1]{}

\title{Colorful Helly-type Theorems for the Volume of Intersections of Convex 
Bodies}

\author{G\'abor Dam\'asdi}
\address{G. D.: Dept. of Computer Science, ELTE E\"otv\"os Lor\'and University, 
Budapest, Hungary}
\email{damasdigabor@caesar.elte.hu}

\author{Vikt\'oria F\"oldv\'ari}
\address{V. F.: Dept. of Geometry, ELTE E\"otv\"os Lor\'and University, 
Budapest, Hungary}
\email{foldvari@math.elte.hu}

\author{M\'arton Nasz\'odi}
\address{M. N.: 
Alfr\'ed R\'enyi Inst. of Math.; 
MTA-ELTE Lend\"ulet Combinatorial Geometry Research Group;
Dept. of Geometry, ELTE E\"otv\"os Lor\'and University, Budapest, Hungary}
\email{marton.naszodi@math.elte.hu}

\begin{document}
\begin{abstract}
We prove the following Helly-type result. Let 
$\mathcal{C}_1,\dots,\mathcal{C}_{3d}$ be finite families of convex bodies in 
$\mathbb{R}^d$. Assume that for any colorful selection of $2d$ sets, 
$C_{i_k}\in 
\mathcal{C}_{i_k}$ for each $1\leq k\leq 2d$ with $1\leq i_1<\dots<i_{2d}\leq 
3d$, the intersection $\bigcap\limits_{k=1}^{2d} C_{i_k}$ is of volume at least 
1. Then there is an $1\leq i \leq 3d$ such that 
$\bigcap\limits_{C\in \mathcal{C}_i} C$ is of volume at least $d^{-O(d^2)}$.
\end{abstract}

\maketitle

\section{Introduction}\label{sec:results}

According to Helly's Theorem, \emph{if the intersection of any $d+1$ members of 
a finite family of convex sets in $\Red$ is non-empty, then the intersection of 
all members of the family is non-empty.} 

A generalization of Helly's Theorem, known as the \textbf{Colorful Helly 
Theorem}, was given by Lov\'asz, and later by Bárány \cite{MR676720}: \emph{If 
$\Ci_1,\dots, \Ci_{d+1}$ are finite families (color classes) of convex sets in 
$\Red$, such 
that for any colorful selection $C_1\in \Ci_1,\dots, C_{d+1}\in \Ci_{d+1}$, the 
intersection $\bigcap\limits_{i=1}^{d+1} C_i$ is non-empty, then for some $j$, 
the intersection $\bigcap\limits_{C\in \Ci_j} C$ is also non-empty.}

Another variant of Helly's Theorem was introduced by B\'ar\'any, Katchalski and 
Pach \cite{BKP82}, whose \textbf{Quantitative Volume Theorem} states the 
following. \emph{Assume that the intersection of any $2d$ members of a finite 
family of convex sets in $\Red$ is of volume at least 1. Then the volume of 
the intersection of all members of the family is of volume at least $c_d$, a 
constant depending on $d$ only.}

They proved that one can take $c_d=d^{-2d^2}$ and conjectured that it should 
hold with $c_d=d^{-cd}$ for an absolute constant $c>0$. It was confirmed  with 
$c_d\approx d^{-2d}$ in \cite{MR3439267}, whose argument was then
refined by Brazitikos \cite{Bra17}, who showed that one may take $c_d\approx 
d^{-3d/2}$. For more on quantitative Helly-type results, see the surveys 
\cite{HW18survey,DGFM19survey}.

In the present paper, we combine the two directions: colorful and quantitative.

\subsection{Ellipsoids and volume}

A well known consequence of John's Theorem (Corollary~\ref{cor:johnvolume}), is 
that any compact convex set $K$ with non-empty interior contains a unique 
ellipsoid $\Ei$ of maximal volume, moreover, $\Ei$ enlarged around its center 
by a factor $d$ contains $K$ (cf. \cite{Ba97}). It follows that the 
volume of the largest ellipsoid contained in $K$ is of volume at least 
$d^{-d}\V(K)$. More precise bounds for this volume ratio are known (cf. 
\cite{Ba97}), but we will not need them.

As shown in \cite[Section~3]{MR3439267}, in the Quantitative Volume Theorem, 
the $d^{-cd}$ factor is sharp up to the absolute constant $c$. In particular, 
for every sufficiently large positive integer $d$, there is a family of convex 
sets satisfying the assumptions of the theorem whose intersection is of volume 
roughly $d^{-d/2}$.

John's Theorem and the fact above yield that bounding the volume of 
intersections and bounding the volume of ellipsoids contained in the 
intersections are essentially equivalent problems: the only difference is a 
multiplicative factor $d^d$ which is of no consequence, unless one wants to 
find the best constants in the exponent. Thus, from this point on, 
we phrase our results in terms of the volume of ellipsoids contained in 
intersections. Its benefit is that this is how in the proofs we actually 
``find volume'': we find ellipsoids of large volume.

\subsection{Main result: few color classes}

Our main result is the following.

\begin{theorem}[Colorful Quantitative Volume Theorem with Ellipsoids -- Few 
Color Classes]\label{thm:colorfulqellipsoid}
Let $\C_1,\linebreak[0]\dots,\C_{3d}$ be finite families of convex bodies in 
$\Red$. Assume that for any colorful selection of $2d$ sets, $C_{i_k}\in 
\C_{i_k}$ for each $1\leq k\leq 2d$ with $1\leq i_1<\dots<i_{2d}\leq 3d$, the 
intersection $\bigcap\limits_{k=1}^{2d} C_{i_k}$ contains an ellipsoid of 
volume at least 1.
Then, there exists an $1\leq i \leq 3d$ such that $\bigcap\limits_{C\in \C_i} 
C$ contains an ellipsoid of volume at least $c^{d^2}d^{-5d^2/2}$ with an 
absolute constant $c\geq0$.
\end{theorem}

We rephrase this theorem in terms of the volume of intersections, as this form 
may be more easily applicable.
\begin{cor}[Colorful Quantitative Volume Theorem -- Few Color 
Classes]\label{cor:colorfulqellipsoid}
Let $\C_1,\linebreak[0]\dots,\C_{3d}$ be finite families of convex bodies in 
$\Red$. 
Assume that for any colorful selection of $2d$ sets, $C_{i_k}\in \C_{i_k}$ for 
each $1\leq k\leq 2d$ with $1\leq i_1<\dots<i_{2d}\leq 3d$, the intersection 
$\bigcap\limits_{k=1}^{2d} C_{i_k}$ is of volume at least 1. 

Then, there exists an $1\leq i \leq 3d$ such that $\V\left(\bigcap\limits_{C\in 
\C_i} 
C\right)\geq c^{d^2}d^{-7d^2/2}$ with an absolute constant $c\geq0$.
\end{cor}

Observe that the smaller the number of color classes in a colorful Helly-type 
theorem, the stronger the theorem is. For example, the Colorful Helly 
Theorem (see the top of the section) is stated with $d+1$ color classes, but it 
is easy to see that it implies the same result with $\ell\geq d+2$ color 
classes, as the last $\ell-(d+1)$ color classes make the assumption of the 
theorem stronger and the conclusion weaker. We note also that the Colorful 
Helly Theorem does not hold with less than $d+1$ color classes, as the number 
$d+1$ cannot be replaced by any smaller number in Helly's Theorem.

The novelty of the proof of Theorem~\ref{thm:colorfulqellipsoid} is the 
following. As we will see later, similar looking statements can be obtained by 
taking the Quantitative Volume Theorem as a ``basic'' Helly-type theorem, and 
combining it with John's Theorem and a combinatorial argument. This approach 
yields results with $d(d+3)/2$ color classes, but does not seem to yield 
results with fewer color classes. In order to achieve that, first, we introduce 
an ordering on the set of ellipsoids, and second, we give a finer geometric 
examination of the situation by comparing the maximum volume ellipsoid of a 
convex body $K$ to other ellipsoids contained in $K$. 

We find it an intriguing question whether one can decrease the number of color 
classes to $2d$ (possibly with an even weaker bound on the volume of the 
ellipsoid obtained), and whether an order $d^{-cd}$ lower bound on the volume 
of 
the ellipsoid 
can be shown.

\subsection{Earlier results and simple observations}

In 1937, Behrend \cite{MR1513119} (see also Section~6.17 of the survey 
\cite{MR0157289} by Danzer, Gr\"unbaum and Klee) proved a planar quantitative 
Helly-type result:
\emph{If the intersection of any 5 members of a finite family of convex sets in 
$\Re^2$
contains an ellipse of area 1, then the intersection of all members of the 
family contains an ellipse of area 1}. We note that, since every convex set in 
$\Re^2$ is the intersection of the half-planes containing it, the result is 
equivalent to the formally weaker statement where the family consists of 
half-planes only. This is the form in which it is stated in \cite{MR0157289}.

In \cite[Section~6.17]{MR0157289}, it is mentioned that John's Theorem 
(Theorem~\ref{John}) should be applicable to extend Behrend's result to 
higher dimensions. We spell out this argument, and present a straightforward 
proof of the following.

\begin{proposition}[Helly-type Theorem with Ellipsoids]\label{ell}
Let $\Ci$ be a finite family of at least $d(d+3)/2$ convex sets in $\Red$, and 
assume that 
for any selection $C_1,\dots, C_{d(d+3)/2}\in \Ci$, the intersection 
$\bigcap\limits_{i=1}^{d(d+3)/2} C_i$ contains an ellipsoid of volume 1. Then  
$\bigcap\limits_{C\in\Ci} C$ also contains an ellipsoid of volume 1. 
\end{proposition}

The number $d(d+3)/2$ is best possible. Indeed, for every dimension $d$, there 
exists a family of $d(d+3)/2$ half-spaces such that the unit ball $\Ball$ is 
the maximum volume ellipsoid contained in their intersection, but $\Ball$ is 
not the maximum volume ellipsoid contained in the intersection of any proper 
subfamily of them. That is, the intersection of any subfamily of $d(d+3)/2-1$ 
members contains an ellipsoid of larger volume than the volume of 
$\Ball$ (which we denote by $\omegad=\V(\Ball)$), and yet, the intersection of 
all members of the family does not contain an ellipsoid of larger volume than 
$\omegad$. This follows from the much stronger result, Theorem~4 in \cite{Gr88} 
by Gruber.

We prove a colorful version of Proposition~\ref{ell}.

\begin{proposition}[Colorful Quantitative Volume Theorem with Ellipsoids -- 
Many Color Classes]\label{colell}
Let $\Ci_1,\dots, \Ci_{d(d+3)/2}$ be finite families of convex 
bodies
in $\Red$, and assume that for any colorful selection 
$C_1\in \Ci_1,\dots, C_{d(d+3)/2}\in \Ci_{d(d+3)/2}$, the intersection 
$\bigcap\limits_{i=1}^{d(d+3)/2} C_i$ contains an ellipsoid of volume 1. Then 
for some $j$, the intersection $\bigcap\limits_{C\in\C_j} C$ contains 
an ellipsoid of volume 1. 
\end{proposition}

According to \cite{algebr} by De Loera et al., a monochromatic Helly-type 
theorem implies a colorful version in a certain \emph{combinatorial setting}.
The novelty of Proposition~\ref{colell} is the \emph{geometric part} of the 
proof, where we introduce an ordering on the family of ellipsoids, and study 
the properties of this ordering (see Section~\ref{sec:ordering}).

Sarkar, Xue and Sober{\'o}n \cite[Corollary~1.0.5]{SaXuSo}, using matroids, 
recently obtained a result involving $d(d+3)/2$ color classes, but with the 
number of selected sets being $2d$.

\begin{proposition}[Sarkar, Xue and Sober{\'o}n \cite{SaXuSo}]\label{thm:SaXuSo}
Let $\C_1,\linebreak[0]\dots,\C_{d(d+3)/2}$ be finite families of convex 
bodies in $\Red$. Assume that for any colorful selection of $2d$ sets, 
$C_{i_k}\in \C_{i_k}$ for each $1\leq k\leq 2d$ with $1\leq 
i_1<\dots<i_{2d}\leq d(d+3)/2$, the intersection $\bigcap\limits_{k=1}^{2d} 
C_{i_k}$ 
contains an ellipsoid of volume at least 1. Then, there exists an $1\leq i \leq 
d(d+3)/2$ such that $\bigcap\limits_{C\in \C_i} C$ has volume at least 
$d^{-O(d)}$.
\end{proposition}

For completeness, in Section~\ref{subsec:SaXuSo}, we sketch a brief argument 
showing that Proposition~\ref{thm:SaXuSo} immediately follows from 
our Proposition~\ref{colell} and the Quantitative Volume Theorem.


The structure of the paper is the following. In Section~\ref{sec:prelim}, we 
introduce some preliminary facts and definitions, notably, an ordering 
on the family of ellipsoids of volume at least 1 that are contained in a convex 
body. Section~\ref{sec:proofs} contains the proofs of our results.

\section{Preliminaries}\label{sec:prelim}

\subsection{John's ellipsoid}
\begin{theorem}[John \cite{Jo48}]\label{John}
Let $K\subset \Red$ be a convex body. Then $K$ contains a unique ellipsoid of 
maximal volume.  This ellipsoid is $\Ball$ if and only if $\Ball\subset K$ and 
there are contact points $u_1,\dots,u_m\in bd(K)\cap bd(\Ball)$ and positive 
numbers $\lambda_1,\dots,\lambda_m$ with $d+1\le m\le \frac{d(d+3)}{2}$ such 
that 
	\begin{equation*}
	\sum\limits_{i=1}^m\lambda_iu_i=0,
		\mbox{ and }
	 I_d=\sum\limits_{i=1}^m\lambda_i u_iu_i^T,
	\end{equation*}
where $I_d$ denotes the $d\times d$ identity matrix and the $u_i$ are column 
vectors.
\end{theorem}

The following is a well known corollary, see \cite[Lecture~3]{Ba97}.
\begin{cor}\label{cor:johnvolume}
 Assume that $\Ball$ is the unique maximal volume ellipsoid contained in a 
convex body $K$ in $\Red$. Then $d\Ball\supseteq K$.
\end{cor}

\subsection{Colorful Helly Theorem}
We recall the Colorful Helly Theorem, as one of its straightforward 
corollaries will be used.

\begin{theorem}[Colorful Helly Theorem, Lov\'asz, 
B\'ar\'any \cite{MR676720}]\label{thm:colorful}
Let $\Ci_1,\dots, \Ci_{d+1}$ be finite families of convex bodies 
in $\Red$, and assume that for any colorful selection $C_1\in \Ci_1,\dots, 
C_{d+1}\in \Ci_{d+1}$, the intersection $\bigcap\limits_{i=1}^{d+1} C_i$ is 
non-empty. 
Then for some $j$, the intersection $\bigcap\limits_{C\in \Ci_j} C$ is also 
non-empty.
\end{theorem}

\begin{cor}\label{cor:chellyklee}
Let $\Ci_1,\dots, \Ci_{d+1}$ be finite families of convex bodies, and $L$ 
a convex body in $\Red$. Assume that for any colorful selection $C_1\in 
\Ci_1,\dots, C_{d+1}\in \Ci_{d+1}$, the intersection 
$\bigcap\limits_{i=1}^{d+1} C_i$ 
contains a translate of $L$.
Then for some $j$, the intersection $\bigcap\limits_{C\in \Ci_j} C$ contains a 
translate 
of $L$. 
\end{cor}

\begin{proof}[Proof of Corollary~\ref{cor:chellyklee}]
We use the following operation, the Minkowski difference of two convex sets $A$ 
and $B$:
\[
A\sim B:=\bigcap_{b\in B} (A-b).
\]
It is easy to see that $A\sim B$ is the set of those vectors $t$ such that 
$B+t\subseteq A$.

By the assumption, for any colorful selection $C_1\in \Ci_1,\dots, C_{d+1}\in 
\Ci_{d+1}$, we have $\bigcap\limits_{i=1}^{d+1} \left(C_i\sim 
L\right)\neq\emptyset$. By Theorem~\ref{thm:colorful}, for some $j$, we have 
$\bigcap\limits_{C\in \Ci_j} \left(C\sim L\right)\neq\emptyset$, and thus, 
$\bigcap\limits_{C\in 
\Ci_j} C$ contains a translate of $L$.  
\end{proof}

\subsection{Lowest ellipsoid}\label{sec:ordering}

We will follow Lovász' idea of the proof of the Colorful Helly 
Theorem. The first step is to fix an ordering of the objects of study. This 
time, we are looking for an ellipsoid and not a point in the intersection, 
therefore we need an \emph{ordering on the ellipsoids}.

For an ellipsoid $\Ei$, we define its \emph{height} as the largest value of the 
orthogonal projection of $\Ei$ on the last coordinate axis, that is,
$\max\{x^Te_d\st x\in\Ei\}$, where $e_d=(0,0,\dots,0,1)^T$.



\begin{lemma}\label{lowest}
Let $C$ be a convex body that contains an ellipsoid of 
volume $\omegad:=\V(\Ball)$. Then there is a unique ellipsoid of volume 
$\omegad$ 
such that every other ellipsoid of volume $\omegad$ in $C$ has larger 
height. Furthermore, if $\tau\in\Re$ denotes the height of this ellipsoid, then 
the largest volume ellipsoid of the convex body $H_{\tau}\cap 
C$ is this ellipsoid, where $H_{\tau}$ denotes the closed 
half-space $H_{\tau}=\{x\in\Red\st x^Te_d\le \tau\}$. 
\end{lemma}
We call this ellipsoid the \emph{lowest ellipsoid} in $C$.

\begin{proof}[Proof of Lemma~\ref{lowest}]
It is not difficult to see that $H_{\tau}\cap C$ does not contain any ellipsoid 
of volume larger than $\omegad$. Indeed, otherwise for a sufficiently small 
$\epsilon>0$, the set $H_{\tau-\epsilon}\cap C$
would contain an ellipsoid of volume equal to $\omegad$, where 
$H_{\tau-\epsilon}$ denotes the closed half-space 
$H_{\tau-\epsilon}=\{x\in\Red\st x^Te_d\le \tau-\epsilon\}$. 

Thus, by Theorem~\ref{John}, $\Ball$ is the 
unique largest volume ellipsoid of $H_{\tau}\cap C$. It follows that $\Ball$ is 
the unique lowest ellipsoid of $C$.
\end{proof}

\subsection{Quantitative Volume Theorem with Ellipsoids}\label{sec:qhelly}

We will rely on the following quantitative Helly theorem.
\begin{theorem}[Quantitative Volume Theorem]\label{thm:qhellyellvol}
Let $C_1,\dots,C_n$ be convex sets in $\Red$. Assume that the intersection of 
any $2d$ of them is of volume at least 1. Then $\V\left(\bigcap\limits_{i=1}^n 
C_i\right)\geq c^dd^{-3d/2}$ with an absolute constant $c>0$.
\end{theorem}

As noted in Section~\ref{sec:results}, it is shown in \cite{MR3439267} that the 
$d^{-3d/2}$ term cannot be improved further than $d^{-d/2}$.

\begin{cor}[Quantitative Volume Theorem with 
Ellipsoids]\label{lem:qhellyell}
Let $C_1,\dots,C_n$ be convex sets in $\Red$. Assume that the intersection of 
any $2d$ of them contains an ellipsoid of volume at least 1. Then 
$\bigcap\limits_{i=1}^n C_i$ contains an ellipsoid of volume at least 
$c^dd^{-5d/2}$ with an absolute constant $c>0$.
\end{cor}

Theorem~\ref{thm:qhellyellvol} was proved by 
B\'ar\'any, Katchalski and Pach \cite{BKP82} with the weaker volume bound 
$d^{-2d^2}$. In \cite{MR3439267}, the volume bound $c^d d^{-2d}$ was shown, and 
this argument was later refined by Brazitikos \cite{Bra17} to obtain the bound 
presented above. An inspection of the argument in \cite{MR3439267} shows that 
Corollary~\ref{lem:qhellyell} holds with the slightly stronger bound 
$c^dd^{-3d/2}$ as well. However, as this constant in the exponent is of no 
consequence, we instead deduce Corollary~\ref{lem:qhellyell} in the form 
presented above from Theorem~\ref{thm:qhellyellvol}.

\begin{proof}[Proof of Corollary~\ref{lem:qhellyell}]
Let $C_1,\dots,C_n$ be convex sets in $\Red$ satisfying the assumptions of 
Corollary~\ref{lem:qhellyell}. In particular, they satisfy the assumptions of 
Theorem~\ref{thm:qhellyellvol}, and hence, $\V\left(\bigcap\limits_{i=1}^n 
C_i\right)\geq c^dd^{-3d/2}$. Finally, Corollary~\ref{cor:johnvolume} yields 
that $\bigcap\limits_{i=1}^n C_i$ contains an ellipsoid of volume at least 
$c^dd^{-5d/2}$ completing the proof of Corollary~\ref{lem:qhellyell}.
\end{proof}

\section{Proofs}\label{sec:proofs}

\subsection{Proof of Proposition~\ref{ell}}
We will prove the following statement, which is clearly equivalent to 
Proposition~\ref{ell}.

\emph{
Assume that the largest volume ellipsoid contained in $\bigcap\limits_{C\in 
\Ci} C$ is of volume $\omegad=\V(\Ball)$. Then there are $d(d+3)/2$ sets in 
$\Ci$ such that the largest volume ellipsoid in their intersection is of volume 
$\omegad$.
}

The problem is clearly affine invariant, and thus, we may assume that the 
largest volume ellipsoid in $\bigcap\limits_{C\in \Ci} C$ is the unit ball 
$\Ball$.

By one direction of Theorem~\ref{John}, there are contact points 
$u_1,\dots,u_m\in bd(\bigcap\limits_{C\in \Ci} C)\cap bd(\Ball)$ and positive 
numbers $\lambda_1,\dots,\lambda_m$ with $d+1\le m\le \frac{d(d+3)}{2}$ 
satisfying the equations in Theorem~\ref{John}. We can choose $C_1,\dots 
C_{m}\in \Ci$ such that $u_i\in bd(C_i)$ for $i=1,\dots,m$.

By the other direction of Theorem~\ref{John}, $\Ball$ is the largest volume 
ellipsoid of $\bigcap\limits_{i=1}^{m} C_i$, completing the proof of 
Proposition~\ref{ell}.

\subsection{Proof of Proposition~\ref{colell}}

\begin{lemma}\label{step3}
Let $C_1,\dots, C_{d(d+3)/2}$ be convex bodies in $\Red$. Assume that 
$K:=\bigcap\limits_{i=1}^{d(d+3)/2}C_i$ contains an ellipsoid of volume 
$\omegad$. Set 
$K_j:=\bigcap\limits_{i=1,i\ne j}^{d(d+3)/2}C_i$, and let $\Ei$ denote the 
lowest ellipsoid in $K$. Then there exists a $j$ such that $\Ei$ is also the 
lowest ellipsoid of $K_j$. 
\end{lemma}
	 
\begin{proof}[Proof of Lemma~\ref{step3}]
Let $\tau$ denote the height of $\Ei$. By Lemma~\ref{lowest}, $\Ei$ is the 
largest volume ellipsoid of $K\cap H_{\tau}$, where $H_{\tau}$ is the 
half-space defined in Lemma~\ref{lowest}.
	
Suppose that $\Ei$ is not the lowest ellipsoid in $K_j$ for every 
$j\in\{1,\dots,d(d+3)/2\}$. Since $\Ei\subset K\subset K_j $, this means that 
each $K_j$ contains a lower ellipsoid than $\Ei$ of volume $\omegad$. Therefore 
we can choose a small $\epsilon>0$ such that $K_j\cap H_{\tau-\epsilon}$ 
contains an ellipsoid of volume $\omegad$ for each $j$, where 
$H_{\tau-\epsilon}$ denotes the closed half-space 
$H_{\tau-\epsilon}=\{x\in\Red\st x^Te_d\le \tau-\epsilon\}$. 
	
Let us consider now the following $\frac{d(d+3)}{2}+1$ sets: $K_1, K_2, \dots, 
K_{d(d+3)/2}, H_{\tau-\epsilon}$. If we take the intersection of 
$\frac{d(d+3)}{2}$ of these sets, we obtain either $K$, or $K_j\cap 
H_{\tau-\epsilon}$ for some $j$. By our assumption, $K$ contains an ellipsoid 
of volume $\omegad$. By the choice of $\epsilon$, we have that $K_j\cap 
H_{\tau-\epsilon}$ also contains an ellipsoid of volume $\omegad$. Hence, 
we can apply Proposition~\ref{ell}, which yields that
$C_1\cap \dots \cap C_{d(d+3)/2}\cap H_{\tau-\epsilon}=K\cap H_{\tau-\epsilon}$ 
also contains an ellipsoid of volume $\omegad$. This contradicts the fact that 
$\Ei$ is the lowest ellipsoid in $K$, and thus, Lemma~\ref{step3} follows.
\end{proof}

We will prove the following statement, which is clearly equivalent to 
Proposition~\ref{colell}.

\emph{
Assume that for every colorful selection $C_1\in \Ci_1,\dots, C_{d(d+3)/2}\in 
\Ci_{d(d+3)/2}$, the intersection $\bigcap\limits_{i=1}^{d(d+3)/2}C_i$
contains an ellipsoid of volume $\omegad$. We will show that 
for some $j$, the intersection $\bigcap\limits_{C\in\C_j} C$ contains 
an ellipsoid of volume $\omegad$.}


By Lemma~\ref{lowest}, we can choose 
the lowest ellipsoid in each of these intersections. Let us denote the set of 
these ellipsoids as ${\mathcal B}$. Since we have finitely many intersections, 
there is a 
highest one among these ellipsoids. Let us denote this ellipsoid by 
$\Ei_{max}$.

\noshow{
 Figure~\ref{main} depicts a very special and trivial case of the theorem. 
 Three of the color classes only contain a single ellipsoid, and the other two 
 contains two parallelograms. In this case ${\mathcal B}$ only contains the 
four dark 
 ellipsoids.  
	
	\begin{figure}[!htbp]
		\begin{center}
			\includegraphics[height=2.3in]{main.png}
			\caption{A very special case of Proposition~\ref{colell}. }
			\label{main}
		\end{center}
	\end{figure}
}

$\Ei_{max}$ is defined by some $C_1\in \Ci_1,\dots, C_{d(d+3)/2}\in 
\Ci_{d(d+3)/2}$. Once again let $K_j=\bigcap\limits_{i=1,i\ne j}^{d(d+3)/2}C_i$ 
and $K=\bigcap\limits_{i=1}^{d(d+3)/2}C_i$. By Lemma~\ref{step3}, there is a 
$j$ such that $\Ei_{max}$ is the lowest ellipsoid in $K_j$. We will show that 
$\Ei_{max}$ lies in every element of $\Ci_j$ for this $j$.

Fix a member $C_0$ of $\Ci_j$. Suppose that 
$\Ei_{max}\not\subset C_0$. Then $\Ei_{max}\not\subset C_0\cap K_j$. By the 
assumption of Proposition~\ref{colell}, $C_0\cap K_j$ contains an ellipsoid of 
volume $\omegad$, since it is the intersection of a colorful selection of sets. 
Since $C_0\cap K_j\subset K_j$, the lowest ellipsoid of $C_0\cap K_j$ is at 
least as high as the lowest ellipsoid of $K_j$. But the unique lowest ellipsoid 
of $K_j$ is $\Ei_{max}$, and $\Ei_{max}\not\subset C_0\cap K_j$. So the lowest 
ellipsoid of $C_0\cap K_j$ lies higher than $\Ei_{max}$. This contradicts that 
$\Ei_{max}$ was chosen to be the highest among the ellipsoids in ${\mathcal 
B}$. So 
$\Ei_{max}\subset C_0$. Since $C_0\in \Ci_j$ was chosen arbitrarily, we obtain 
that $\Ei_{max}\subset \bigcap\limits_ {C\in \Ci_j} C$, completing the proof of 
Proposition~\ref{colell}.

\subsection{Proof of Proposition~\ref{thm:SaXuSo}}\label{subsec:SaXuSo}
Consider an arbitrary colorful selection of $d(d+3)/2$ convex bodies. By 
Corollary~\ref{lem:qhellyell}, their intersection contains an ellipsoid of 
volume 
at least $c^dd^{-5d/2}$. It follows immediately from Proposition~\ref{colell}, 
that 
the intersection of one of the color classes contains an ellipsoid of volume at 
least $c^dd^{-5d/2}$, completing the proof of Proposition~\ref{thm:SaXuSo}.

\subsection{Proof of Theorem~\ref{thm:colorfulqellipsoid}}

We will prove the following statement, which is clearly equivalent to 
Theorem~\ref{thm:colorfulqellipsoid}.

\emph{
Assume that the intersection of all colorful selections of $2d$ sets contains 
an ellipsoid of volume at least $\omegad$. Then, there is an $1\leq i \leq 3d$ 
such that $\bigcap\limits_{C\in \C_i} C$ contains an ellipsoid of volume at 
least $c^{d^2}d^{-5d^2/2}\omegad$ with an absolute constant $c\geq0$.
}

\begin{lemma}\label{lem:hellyklee}
Assume that $\Ball$ is the largest volume ellipsoid contained in the convex set 
$C$ in  $\Red$. Let $\Ei$ be another ellipsoid in $C$ of 
volume at least $\delta\omegad$ with $0<\delta<1$. Then there is a translate of 
$\frac{\delta}{d^{d-1}} 
\Ball$ which is contained in $\Ei$.
\end{lemma}
\begin{proof}[Proof of Lemma~\ref{lem:hellyklee}]
If the length of all $d$ semi-axes $a_1,\dots,a_d$ of $\Ei$ are at least 
$\lambda$ for some $\lambda>0$, then clearly, $\lambda \Ball+c \subset 
\Ei$, where $c$ denotes the center of $\Ei$. We will show that all the 
semi-axes are long enough.

By Corollary~\ref{cor:johnvolume}, $\Ei\subset C \subset d \Ball$. Therefore, 
$a_i\leq d$ for 
every $i=1,\dots,d$. Since the volume of $\Ei$ is 
$a_1\cdots a_d \omegad\geq \delta\omegad$, we have  $a_i\geq 
\frac{\delta}{d^{d-1}}$ for every $i=1,\dots,d$, completing the proof of 
Lemma~\ref{lem:hellyklee}.
\end{proof}

Consider the lowest ellipsoid in the intersection of all colorful selections of 
$2d-1$ sets. We may assume that the highest one of these ellipsoids is $\Ball$. 
By possibly changing the indices of the families, we may assume that the 
selection is $C_1\in\C_1,\dots,C_{2d-1}\in\C_{2d-1}$. We call $\C_{2d}, 
\C_{2d+1},\dots,\C_{3d}$ the \emph{remaining families}.

Consider the half-space $H_{1}=\{x\in\Red\st x^Te_d\le 1\}\supset \Ball$. By 
Lemma~\ref{lowest}, $\Ball$ is the largest volume ellipsoid contained in $M := 
C_1\cap\dots\cap C_{2d-1} \cap H_1$.

Next, take an arbitrary colorful selection 
$C_{2d}\in\C_{2d}, C_{2d+1}\in\C_{2d+1},\dots,C_{3d}\in\C_{3d}$ of the 
remaining $d+1$ families. We 
claim that the intersection of any $2d$ sets of
\[C_1,\dots, C_{2d-1}, H_1, C_{2d},\dots, C_{3d}\]
contains an ellipsoid of volume at least $\omegad$. 
Indeed, if $H_1$ is not among those $2d$ sets, then our assumption ensures 
this. 
If $H_1$ is among them, then by the choice of $H_1$, the claim holds.

Therefore, by Theorem~\ref{lem:qhellyell}, the 
intersection 
\[
 \bigcap\limits_{i=1}^{3d} C_i\cap H_1
\]
contains an ellipsoid $\Ei$ of volume at least 
$\delta\omegad$, where $\delta:=c^dd^{-3d/2}$. Clearly, $\Ei\subset M$.

Since $\Ball$ is the maximum volume ellipsoid contained in $M$, by Lemma 
\ref{lem:hellyklee}, we 
have that there is a translate of $\frac{\delta}{d^{d-1}} \Ball$ which is 
contained in $\Ei$ and thus in $\bigcap\limits_{i=2d}^{3d} C_i$.

Thus, we have shown that any colorful selection 
$C_{2d}\in\C_{2d}, C_{2d+1}\in\C_{2d+1},\dots,C_{3d}\in\C_{3d}$ of the 
remaining $d+1$ families, 
$\bigcap\limits_{i=2d}^{3d} C_i$ contains a translate of the same convex body 
$c^dd^{-5d/2}\Ball$. It follows from Corollary~\ref{cor:chellyklee} 
that there is 
an index $2d\leq i \leq 3d$ such that $\bigcap\limits_{C\in \C_i} C$ contains a 
translate of $c^dd^{-5d/2}\Ball$, which is an ellipsoid of 
volume $c^{d^2}d^{-5d^2/2}\omegad$, finishing the proof of 
Theorem~\ref{thm:colorfulqellipsoid}.

\subsection{Proof of Corollary~\ref{cor:colorfulqellipsoid}}
By Corollary~\ref{cor:johnvolume}, the volume of the largest ellipsoid in a 
convex body is at least $d^{-d}$ times the volume of the body.
Corollary~\ref{cor:colorfulqellipsoid} now follows immediately from 
Theorem~\ref{thm:colorfulqellipsoid}.

\newcommand\invisiblesection[1]{%
  \refstepcounter{section}%
  \addcontentsline{toc}{section}{\protect\numberline{\thesection}#1}%
  \sectionmark{#1}}

\invisiblesection{Acknowledgement}
\subsection*{Acknowledgement}
Some of the results were part of GD's master's thesis. 
Part of the research was carried out while MN was a member of J\'anos Pach's 
chair of DCG at EPFL, Lausanne, which was supported by Swiss National Science 
Foundation Grants 200020-162884 and 200021-165977.
GD and MN also acknowledge the support of the National Research, Development 
and Innovation Fund grant K119670 as well as that of the MTA-ELTE Lend\"ulet 
Combinatorial Geometry Research Group.


\bibliographystyle{halpha-abbrv}
\bibliography{biblio}

\end{document}